\documentclass[a4paper,11pt]{amsart}
\usepackage[matrix,arrow,tips,curve]{xy}
\usepackage[english]{babel}
\usepackage{amsmath}
\usepackage{amssymb}
\usepackage{tikz}
\usepackage{mathrsfs}
\usepackage{enumerate}
\usepackage{graphicx}
\usepackage{hyperref}
\usepackage{verbatim}
\usetikzlibrary{trees}
\usetikzlibrary{arrows}

\input{xy}
\xyoption{all}
\newtheorem{thm}{Theorem}[section]

\newtheorem{Proposition}[thm]{Proposition}
\newtheorem{Corollary}[thm]{Corollary}

\newtheorem*{thm*}{Theorem}

\theoremstyle{definition}

\newtheorem{Remark}[thm]{Remark}

\DeclareMathOperator{\rk}{rank}
\DeclareMathOperator{\brk}{\underline{rank}}

\DeclareMathOperator{\expdim}{expdim}

\DeclareMathOperator{\Sym}{Sym}
\renewcommand{\sec}{\mathbb{S}ec}

\hypersetup{pdfpagemode=UseNone}
\hypersetup{pdfstartview=FitH}
		
\begin{document}
\title{Effective identifiability criteria for tensors and polynomials}

\author[Alex Massarenti]{Alex Massarenti}
\address{\sc Alex Massarenti\\
Universidade Federal Fluminense (UFF)\\
Rua M\'ario Santos Braga\\
24020-140, Niter\'oi,  Rio de Janeiro\\ Brazil}
\email{alexmassarenti@id.uff.br}

\author{Massimiliano Mella}
\address{\sc Massimiliano Mella\\ Dipartimento di Matematica e Informatica, Universit\`a di Ferrara, Via Machiavelli 35, 44121 Ferrara, Italy}
\email{mll@unife.it}

\author{Giovanni Staglian\`o}
\address{\sc Giovanni Staglian\`o}
\email{giovannistagliano@gmail.com}

\date{\today}

\begin{abstract}
A tensor $T$, in a given tensor space, is said to be $h$-identifiable if it admits a unique decomposition as a sum of $h$ rank one tensors. A criterion for $h$-identifiability is called effective if it is satisfied in a dense, open subset of the set of rank $h$ tensors. In this paper we give effective $h$-identifiability criteria for a large class of tensors.  
We then improve these criteria for some symmetric tensors.  For instance, this allows us to give a complete set of effective identifiability criteria for ternary quintic polynomial. Finally, we implement our identifiability algorithms in Macaulay2.
\end{abstract}

\maketitle
\tableofcontents

\section{Introduction}\label{intro}
A tensor rank-1 decomposition of a tensor $T$, lying in a given tensor space over a field $k$, is an expression of the type 
\stepcounter{thm}
\begin{equation}\label{eq1gen}
T = \lambda_1 U_1+...+\lambda_h U_{h}
\end{equation}
where the $U_i$'s are linearly independent rank one tensors, $\lambda_i\in k^*$, and $k$ is either the real or complex field. The \textit{rank} of $T$, denoted by  $\rk(T)$, is the minimal positive integer $h$ such that $T$ admits a decomposition as in (\ref{eq1gen}). 

Tensor decomposition problems and techniques are of relevance in both pure and applied mathematics. For instance, tensor decomposition algorithms have applications in psycho-metrics, chemometrics, signal processing, numerical linear algebra, computer vision, numerical analysis, neuroscience and graph analysis \cite{BK09}, \cite{CM96}, \cite{CGLM08}, \cite{LO15}, \cite{MR13}. In pure mathematics tensor decomposition issues naturally arise in constructing and studying moduli spaces of all possible additive decompositions of a general tensor into a given number of rank one tensors \cite{Do04}, \cite{DK93}, \cite{MM13}, \cite{Ma16}, \cite{RS00}, \cite{TZ11}, 

We say that a tensor rank-1 decomposition has the \textit{generic identifiability property} if the expression (\ref{eq1gen}) is unique, up to permutations and scaling of the factors, on a dense open subset of the set of tensors admitting an expression as in (\ref{eq1gen}). This uniqueness property is useful in several application, we refer to \cite{COV16} for an account.

Given a tensor rank-1 decomposition of length $h$ as in (\ref{eq1gen}) the problem of \textit{specific identifiability} consists in proving that such a decomposition is unique. Following \cite{COV16} we call an algorithm for specific identifiability \textit{effective} if it is sufficient to prove identifiability on a dense open subset of the set of tensors admitting a decomposition as in (\ref{eq1gen}). Therefore, an algorithm is effective if its constraints are satisfied generically, in other words if the same algorithm proves generic identifiability as well.

In this paper we consider symmetric tensors, mixed skew-symmetric tensors, and mixed symmetric tensors. The corresponding rank-1 tensors are parametrized by Veronese varieties, Segre-Grassmann varieties, and Segre-Veronese varieties respectively. We provide $h$-identifiability effective criteria for these spaces, under suitable numerical assumptions on $h$. Our algorithm are based on the existence of suitable flattenings of a given tensor admitting a decomposition as in (\ref{eq1gen}). We would like to stress that we do not need to know an explicit decomposition but just the fact that such a decomposition exists. 

Recall that the border rank $\brk(T)$ of a tensor $T$ is the smallest integer $r>0$ such that $T$ is in the Zariski closure, in the tensor space where $T$ belongs, of the set of tensors of rank $r$. In particular $\brk(T)\leq \rk(T)$. Roughly speaking, our methods require that suitable linear spaces, defined in terms of flattenings, intersect the relevant varieties parametrizing rank one tensors in a zero dimension scheme of a given length. Such a zero dimensional scheme is not required to be reduced and then our criteria can be applied also in border rank identifiability problems, see Remark \ref{bordR}.

Symmetric tensors can also be interpreted  as  homogeneous polynomials. By rephrasing (\ref{eq1gen}) in the symmetric case we say that a polynomial rank-1 decomposition of a homogeneous degree $d$ polynomial $F\in k[x_0,...,x_n]_d$ is an expression of the type 
\stepcounter{thm}
\begin{equation}\label{eq1}
F = \lambda_1 L_1^d+...+\lambda_h L_{h}^d
\end{equation}
where $L_i$ are linearly independent degree 1 polynomials, $\lambda_i\in k^*$, and $k$ is either the
real or complex field. 
Let $h(n,d)$ be the minimum integer such that a general $F\in
k[x_0,...,x_n]_d$ admits a decomposition as in (\ref{eq1}). The number
$h(n,d)$ has been determined in \cite{AH95} and $h(n,d)$-identifiability very seldom holds \cite{Me06}, \cite{Me09}, \cite{GM16}.
Indeed, by \cite[Theorem 1]{GM16} a general polynomial $F\in k[x_{0},...,x_{n}]_{d}$ is $h(n,d)$-identifiable only in the following cases:
\begin{itemize}
\item[-] $n = 1$, $d = 2m+1$, $h(n,d)=m$ \cite{Sy04},
\item[-] $n = d = 3$, $h(3,3) = 5$ \cite{Sy04},
\item[-] $n = 2$, $d = 5$, $h(2,5) = 7$ \cite{Hi88}.
\end{itemize}
In Theorem \ref{hi} we provide effective $h$-identifiability criteria for these polynomials and combined with the previous results this furnishes a complete set of identifiability criteria for these, and few more, polynomials. We would like to stress that the identifiability criteria in Theorem \ref{hi} give new proves of the uniqueness of the decomposition for the general polynomial in the three cases listed above. Finally, in Section \ref{mac2} we implemented our identifiability algorithms in Macaulay2 \cite{Mc2}.

\subsection*{Acknowledgments}
The authors are members of the Gruppo Nazionale per le Strutture Algebriche, Geometriche e le loro Applicazioni of the Istituto Nazionale di Alta Matematica "F. Severi" (GNSAGA-INDAM). MM is partially supported by PRIN "Geometry of Algebraic Varieties" (MIUR 2015EYPTSB\_005). 

\section{Tensors and flattenings}
Let $\underline{n}=(n_1,\dots,n_p)$ and $\underline{d} = (d_1,\dots,d_p)$ be two $p$-uples of positive integers.
Set 
$$d=d_1+\dots+d_p,\ n=n_1+\dots+n_p,\ {\rm and}\
N(\underline{n},\underline{d})=\prod_{i=1}^p\binom{n_i+d_i}{n_i}.$$ 

Let $V_1,\dots, V_p$ be vector spaces of dimensions $n_1+1\leq n_2+1\leq \dots \leq n_p+1$, and consider the product
$$
\mathbb{P}^{\underline{n}} = \mathbb{P}(V_1^{*})\times \dots \times \mathbb{P}(V_p^{*}).
$$
The line bundle 
$$
\mathcal{O}_{\mathbb{P}^{\underline{n}} }(d_1,\dots, d_p)=\mathcal{O}_{\mathbb{P}(V_1^{*})}(d_1)\boxtimes\dots\boxtimes \mathcal{O}_{\mathbb{P}(V_1^{*})}(d_p)
$$
induces an embedding
$$
\begin{array}{cccc}
\sigma\nu_{\underline{d}}^{\underline{n}}:
&\mathbb{P}(V_1^{*})\times \dots \times \mathbb{P}(V_p^{*})& \longrightarrow &
\mathbb{P}(\Sym^{d_1}V_1^{*}\otimes\dots\otimes \Sym^{d_p}V_p^{*})
=\mathbb{P}^{N(\underline{n},\underline{d})-1},\\
      & (\left[v_1\right],\dots,\left[v_p\right]) & \longmapsto & [v_1^{d_1}\otimes\dots\otimes v_p^{d_p}]
\end{array}
$$ 
where $v_i\in V_i$.
We call the image 
$$
SV_{\underline{d}}^{\underline{n}}= \sigma\nu_{\underline{d}}^{\underline{n}}(\mathbb{P}^{\underline{n}} ) \subset \mathbb{P}^{N(\underline{n},\underline{d})-1}
$$ 
a \textit{Segre-Veronese variety}. It is a smooth  variety of dimension $n$ and degree 
$\frac{(n_1+\dots +n_p)!}{n_1!\dots n_p!}d_1^{n_1}\dots d_p^{n_p}$ in $\mathbb{P}^{N(\underline{n},\underline{d})-1}$. 

When $p = 1$, $SV_{d}^{n}$ is a Veronese variety. In this case we write $V_d^n$ for $SV_{d}^{n}$, and $v_{d}^{n}$ for the Veronese embedding.
When $d_1 = \dots = d_p = 1$, $SV_{1,\dots,1}^{\underline{n}}$ is a Segre variety. 
In this case we write $S^{\underline{n}}$ for $SV_{1,\dots,1}^{\underline{n}}$, and  $\sigma^{\underline{n}}$ for the Segre embedding.
Note that 
$$
\sigma\nu_{\underline{d}}^{\underline{n}}=\sigma^{\underline{n}'}\circ \left(\nu_{d_1}^{n_1}\times \dots \times \nu_{d_p}^{n_p}\right),
$$
where $\underline{n}'=(N(n_1,d_1)-1,\dots,N(n_p,d_p)-1)$.

Similarly, given a $p$-uple of $k$-vector spaces $(V_1^{n_1},...,V_{p}^{n_p})$ and  $p$-uple of positive integers $\underline{d} = (d_1,...,d_p)$ we may consider the Segre-Pl\"ucker embedding 
$$
\begin{array}{cccc}
\sigma\pi_{\underline{d}}^{\underline{n}}:
&Gr(d_1,n_1)\times \dots \times Gr(d_p,n_p)& \longrightarrow &
\mathbb{P}(\bigwedge^{d_1}V_1^{n_1}\otimes\dots\otimes \bigwedge^{d_p}V_p^{n_p})
=\mathbb{P}^{N(\underline{n},\underline{d})-1},\\
      & (\left[H_1\right],\dots,\left[H_p\right]) & \longmapsto & [H_1\otimes\dots\otimes H_p]
\end{array}
$$ 
where $N(\underline{n},\underline{d}) = \prod_{i=1}^p\binom{n_i}{d_i}$. We call the image 
$$
SG_{\underline{d}}^{\underline{n}}= \sigma\pi_{\underline{d}}^{\underline{n}}(Gr(d_1,n_1)\times \dots \times Gr(d_p,n_p)) \subset \mathbb{P}^{N(\underline{n},\underline{d})}
$$ 
a \textit{Segre-Grassmann variety}.

The \textit{$h$-secant variety} $\sec_{h}(X)$, of an irreducible, non-degenerate $n$-dimensional variety $X\subset\mathbb{P}^N$, is the Zariski closure of the union of the linear spaces spanned by collections of $h$ points on $X$. The \textit{expected dimension} of $\sec_{h}(X)$ is
$$\expdim(\sec_{h}(X)):= \min\{nh+h-1,N\}.$$
However, the actual dimension of $\sec_{h}(X)$ might be smaller than the expected one. Indeed, this happens when through a general point of $\sec_h(X)$ there are infinitely many $(h-1)$-planes $h$-secant to $X$. We will say that $X$ is \textit{$h$-defective} if $\dim(\sec_{h}(X)) < \expdim(\sec_{h}(X))$.

% In \cite{AH95} J. Alexander and A. Hirshowitz classified secant defective Veronese varieties. Indeed, they proved that, except for the double Veronese embedding which is almost always defective, the degree $d$ Veronese embedding of $\mathbb{P}^n$ is not $h$-defective, with the following exceptions: $(d,n,h)\in\{(4,2,5),(4,3,9),(3,4,7),(4,4,14)\}$. Therefore, with these exceptions the minimum number $h$ of linear forms needed to write a general $F\in k[x_0,...,x_n]_d$ as sum of powers of $h$ linear forms is the number
% $$\war(F)=\lceil\frac{1}{n+1}\binom{n+d}{d}\rceil$$
% called the \textit{Waring rank} of $F$. In Theorem \ref{hi} we will study the identifiability problem in the case $h = \war(F)$.

The following remark was the starting point of the investigation in \cite{MM13}.
\begin{Remark}\label{pd} If a polynomial $F\in k[x_0,...,x_n]_d$
  admits a decomposition as in (\ref{eq1}) then $F\in
  \sec_h(V_{d}^{n})$, and conversely a general $F\in
  \sec_h(V_{d}^{n})$ can be written as in (\ref{eq1}). If $F = \lambda_1L^{d}_{1}+...+\lambda_hL^{d}_{h}$ is a decomposition then the partial derivatives of order $s$ of $F$ can be decomposed as a linear combination of $L^{d-s}_{1},...,L^{d-s}_{h}$ as well.

These partial derivatives are $\binom{n+s}{n}$ homogeneous polynomials of degree $d-s$ spanning a linear space $H_{\partial,s}\subseteq \mathbb{P}(k[x_0,...,x_n]_{d-s})$. Therefore, the linear space $\left\langle L_1^{d-s},\dots,L_h^{d-s}\right\rangle$ contains $H_{\partial,s}$.
\end{Remark}

Our first aim is to generalize Remark~\ref{pd} to tensors. The natural
tools to replace partial derivatives are flattenings.

\subsection{Flattenings}\label{flat}
Let $V_1,...,V_{p}$ be $k$-vector spaces of finite dimension, and consider the tensor product $V_1\otimes ...\otimes V_{p} = (V_{a_1}\otimes ...\otimes V_{a_s})\otimes (V_{b_1}\otimes ...\otimes V_{b_{p-s}})= V_{A}\otimes V_{B}$ with $A\cup B = \{1,...,p\}$, $B = A^c$. Then we may interpret a tensor 
$$T \in V_1\otimes ...\otimes V_p = V_{A}\otimes V_{B}$$
as a linear map $\widetilde{T}:V_{A}^{*}\rightarrow V_{A^c}$. Clearly, if the rank of $T$ is at most $r$ then the rank of $\widetilde{T}$ is at most $r$ as well. Indeed, a decomposition of $T$ as a linear combination of $r$ rank one tensors yields a linear subspace of $V_{A^c}$, generated by the corresponding rank one tensors, containing $\widetilde{T}(V_A^{*})\subseteq V_{A^c}$. The matrix associated to the linear map $\widetilde{T}$ is called an \textit{$(A,B)$-flattening} of $T$.    

In the case of mixed tensors we can consider the embedding
$$\Sym^{d_1}V_1\otimes ...\otimes \Sym^{d_p}V_p\hookrightarrow V_A\otimes V_B$$
where $V_A = \Sym^{a_1}V_1\otimes ...\otimes \Sym^{a_p}V_p$,
$V_B=\Sym^{b_1}V_1\otimes ...\otimes\Sym^{b_p}V_p$, with $d_i =
a_i+b_i$ for any $i = 1,...,p$. In particular, if $n = 1$ we may
interpret a tensor $F\in \Sym^{d_1}V_1$ as a degree $d_1$ homogeneous
polynomial on $\mathbb{P}(V_1^*)$. In this case the matrix associated
to the linear map $\widetilde{F}:V_A^*\rightarrow V_B$ is nothing but
the $a_1$-th \textit{catalecticant matrix} of $F$, that is the matrix
whose lines are the coefficient of the partial derivatives of order
$a_1$ of $F$. This identifies the linear space $H_{\partial,s}$ in
Remark \ref{pd} with $\mathbb{P}(\widetilde{F}(V_A^*))\subseteq
\mathbb{P}(V_B)$, where $a_1 = s$, $b_1 = d-a_1 = d-s$. 

Similarly, by considering the inclusion 
$$\bigwedge^{d_1}V_1\otimes ...\otimes \bigwedge^{d_p}V_p\hookrightarrow V_A\otimes V_B$$
where $V_A = \bigwedge^{a_1}V_1\otimes ...\otimes \bigwedge^{a_p}V_p$, $V_B=\bigwedge^{b_1}V_1\otimes ...\otimes\bigwedge^{b_p}V_p$, with $d_i = a_i+b_i$ for any $i = 1,...,p$, we get the so called \textit{skew-flattenings}. We refer to \cite{La12} for details on the subject.

\section{Effective identifiability}

In this section we give $h$-identifiability criteria for
tensors, and we derive effective $h$-identifiability criteria, under some constraints on $h$.
 
\begin{Proposition}\label{prop1gen}
Let $T\in \Sym^{d_1}V_1\otimes ...\otimes \Sym^{d_n}V_n$ be a tensor
admitting a decomposition $T = \sum_{i=1}^h\lambda_i U_i$ as in~(\ref{eq1gen}). Fix an $(A,B)$-flattening $\widetilde{T}:V_A^*\rightarrow V_B$ of $T$ such that $\dim(V_A^*)\geq h$, and assume that
\begin{itemize}
\item[i)] the linear space $\mathbb{P}(\widetilde{T}(V_A^*))$ has dimension $h-1$,
\item[ii)] $\dim(\mathbb{P}(\widetilde{T}(V_A^*))\cap SV_{\underline{b}}^{\underline{n}}) = 0$,
\item[iii)]  $\deg(\mathbb{P}(\widetilde{T}(V_A^*))\cap SV_{\underline{b}}^{\underline{n}}) = h$.
\end{itemize}
where $\underline{b} = (b_1,...,b_n)$. Then $T$ is $h$-identifiable and it has rank $h$. 

% Furthermore the analogous statement, obtained by substituting the $(A,B)$-flattening with an $(A,B)$-skew flattening and the Segre-Veronese variety $SV_{\underline{b}}^{\underline{n}}$ with the Segre-Grassmann variety $SG_{\underline{b}}^{\underline{n}}$, holds for mixed skew-symmetric tensors. 

In particular, if $F\in k[x_0,...,x_n]_d$ is a polynomial admitting a decomposition $F = \sum_{i=1}^h\lambda_iL_i^d$, $s$ is an integer such that ${n+s\choose n}\geq h>{n+s-1\choose n} $, and
\begin{itemize}
\item[i)] the linear space $H_{\partial,s}$ generated by the partial derivatives of order $s$ of $F$ has dimension $h-1$,
\item[ii)] $\dim(H_{\partial,s}\cap V_{d-s}^n) = 0$,
\item[iii)]  $\deg(H_{\partial,s}\cap V_{d-s}^n) = h$.
\end{itemize}
Then $F$ is $h$-identifiable and it has  rank $h$.  
\end{Proposition}
\begin{proof}
Assume that $T = \sum_{i=1}^h\lambda_iU_i = \sum_{i=1}^h\mu_i V_i$
admits two different decompositions. Since
$\dim(\mathbb{P}(\widetilde{T}(V_A^*))) = h-1$ by Section \ref{flat}
we have $\mathbb{P}(\widetilde{T}(V_A^*)) = \left\langle
  \widetilde{U}_1,...,\widetilde{U}_h\right\rangle = \left\langle
  \widetilde{V}_1,,...,\widetilde{V}_h\right\rangle$, where
$\widetilde{U}_i, \widetilde{V}_i$ are the rank one tensors in
$\mathbb{P}(V_B)$ induced by $U_i$ and $V_i$ respectively. Hence there
are at least $h+1$ points in the intersection
$\mathbb{P}(\widetilde{T}(V_A^*))\cap
SV_{\underline{b}}^{\underline{n}}$, contradicting iii).
\end{proof} 

Next, we check when the conditions in Proposition \ref{prop1gen} define effective criteria.

\begin{Proposition}\label{pro2gen}
The criterion in Proposition \ref{prop1gen} is effective when
$N(\underline{n},\underline{b})>h+\dim(SV_{\underline{b}}^{\underline{n}})$
in the mixed symmetric case. In particular, in the symmetric case the criterion is effective when ${n+d-s\choose n}>h+n$.
\end{Proposition}
\begin{proof}
Let $[T]\in \sec_h(SV_{\underline{d}}^{\underline{n}})$ be a general
point. Assume that $\dim (\mathbb{P}(\widetilde{T}(V_A^*)))\leq
h-2$. This condition forces the $(A,B)$-flattening matrix to have rank
at most $h-1$. On the other hand, by \cite[Proposition 4.1]{SU00} these minors do not vanish on $\sec_h(SV_{\underline{d}}^{\underline{n}})$ and therefore define a closed subset of $\sec_h(SV_{\underline{d}}^{\underline{n}})$. To conclude observe that by the Trisecant Lemma \cite[Proposition 2.6]{CC02}, the general $h$-secant $(h-1)$-linear space intersects $SV_{\underline{b}}^{\underline{n}}$ in $h$ points as long as $N(\underline{n},\underline{b})>h+n$.
\end{proof}

We may slightly improve Proposition \ref{pro2gen}, under suitable numerical assumption.

\begin{Proposition}\label{pro1.1gen}
Let $T\in \Sym^{d_1}V_1\otimes ...\otimes \Sym^{d_p}V_p$ be a tensor admitting a decomposition $T = \sum_{i=1}^h\lambda_i U_i$. Fix an $(A,B)$-flattening $\widetilde{T}:V_A^*\rightarrow V_B$ of $T$ such that $\dim(V_A^*)\geq h$, and assume that
\begin{itemize}
\item[i)] the linear space $\mathbb{P}(\widetilde{T}(V_A^*))$ has dimension $h-1$,
\item[ii)] $\dim(\mathbb{P}(\widetilde{T}(V_A^*))\cap SV_{\underline{b}}^{\underline{n}}) = 0$,
\item[iii)]  $h+n=N(\underline{n},\underline{b})$,
\item[iv)] $\deg(SV_{\underline{b}}^{\underline{n}})\leq h+1$, 
\item[v)]  $\deg(\langle [U_1],\ldots,[U_h]\rangle\cap SV_{\underline{d}}^{\underline{n}})=h$.
\end{itemize}
Then $T$ is $h$-identifiable and the criterion is effective. 

% Again the analogous statement, obtained by substituting the $(A,B)$-flattening with an $(A,B)$-skew flattening and the Segre-Veronese variety $SV_{\underline{b}}^{\underline{n}}$ with the Segre-Grassmann variety $SG_{\underline{b}}^{\underline{n}}$, holds for mixed skew-symmetric tensors. 

In particular, in the symmetric case we have the following. Let $F\in k[x_0,...,x_n]_d$ be a polynomial admitting a decomposition $F = \sum_{i=1}^h\lambda_iL_i^d$.  Fix an integer $s$ such that ${n+s\choose n}\geq h>{n+s-1\choose n} $. Assume that:
\begin{itemize}
\item[i)] the linear space $H_{\partial,s}$ generated by the partial derivatives of order $s$ of $F$ has dimension $h-1$,
\item[ii)] $\dim(H_{\partial,s}\cap V_{d-s}^n) = 0$,
\item[iii)]  $h+n={n+d-s\choose n}$,
\item[iv)] $(d-s)^n\leq h+1$,
\item[v)]  $\deg(\langle[L_1^d],\ldots,[L_h^d]\rangle\cap V_d^n)=h$.
\end{itemize}
Then $F$ is $h$-identifiable and the criterion is effective.
\end{Proposition}
\begin{proof}
Assume that $T = \sum_{i=1}^h\lambda_i U_i = \sum_{i=1}^h\mu_i V_i$
admits two different decompositions. Since
$\dim(\mathbb{P}(\widetilde{T}(V_A^*))) = h-1$ by Section \ref{flat}
we have $\mathbb{P}(\widetilde{T}(V_A^*)) = \left\langle
  \widetilde{U}_1,...,\widetilde{U}_h\right\rangle = \left\langle
  \widetilde{V}_1,...,\widetilde{V}_h\right\rangle$, where
$\widetilde{U}_i, \widetilde{V}_i$ are the rank one tensors in
$\mathbb{P}(V_B)$ induced by $U_i$ and $V_i$ respectively. Assumptions
ii), iii), and iv) show that $\mathbb{P}(\widetilde{T}(V_A^*))$ intersects $SV_{\underline{b}}^{\underline{n}}$ in at most $h+1$ points. Therefore, without loss of generality we may assume that $U_i=V_i$, for $i=1,\ldots,h-1$. By construction we have
$$\langle V_1,\ldots,V_h\rangle=\langle V_1,\ldots,V_{h-1}, F\rangle=\langle U_1,\ldots,U_{h-1}, F\rangle=\langle U_1,\ldots,U_{h}\rangle$$ 
hence $\deg(\langle U_1,\ldots,U_h\rangle\cap SV_{\underline{d}}^{\underline{n}})\geq h+1$ contradicting assumption v). The criterion is effective again by the Trisecant Lemma \cite[Proposition 2.6]{CC02}.
\end{proof}

\begin{Remark}
Propositions, \ref{prop1gen}, \ref{pro2gen}, \ref{pro1.1gen} can be easily extended to the skew symmetric case, using the skew-flattenings in Section \ref{flat}, and the Segre-Grassmann variety instead of the Segre-Veronese variety. We leave the details to the reader.
\end{Remark}

Next, we work out our criterion in some interesting cases, for the
readers convenience we report also the skew symmetric case.

\begin{Corollary}\label{c1}
Let us consider the tensor space $\Sym^{d_1}V_1^{n}\otimes...\otimes \Sym^{d_p}V_p^{n}$, and set $m_i = \lfloor \frac{d_i}{2}\rfloor$. If 
$$h< \prod_{i=1}^p\binom{n-1+m_i}{n-1}-p(n-1)$$ 
then the criterion in Proposition \ref{prop1gen} is effective, while for tensors in $\bigwedge^{d_1}V_1^{n}\otimes...\otimes \bigwedge^{d_p}V_p^{n}$ criterion in Proposition \ref{prop1gen} is effective when
$$h< \prod_{i=1}^p\binom{n}{m_i}-\prod_{i=1}^p m_i(n-m_i).$$
Now, consider $V_1^n\otimes ....\otimes V_p^n$ and set $m = \lfloor \frac{p}{2}\rfloor$. If
$$h< n^m-m(n-1)$$
then the criterion in Proposition \ref{prop1gen} is effective.

Finally, let $V_1^{n_1}\otimes ....\otimes V_p^{n_p}$ be an unbalanced product, that is $n_1> 1 +\prod_{i=2}^p n_i-\sum_{i=2}^p(n_i-1)$. If
$$h<\prod_{i=2}^p n_i -\sum_{i=2}^p(n_i-1)$$
then the criterion in Proposition \ref{prop1gen} is effective.
\end{Corollary}
\begin{proof}
In the mixed symmetric case consider the flattening 
$$\left(\bigotimes_{i=1}^p\Sym^{\lceil \frac{d_i}{2}\rceil}V_i^n\right)^{*}\rightarrow \bigotimes_{i=1}^p\Sym^{\lfloor \frac{d_i}{2}\rfloor}V_i^n$$
while in the mixed skew-symmetric case it is enough to consider the analogous skew-flattening.

Similarly, in the second case we choose the flattening 
$$\left(\bigotimes_{i=1}^{\lceil\frac{p}{2} \rceil}V_i^n\right)^{*}\rightarrow \bigotimes_{i=\lceil\frac{p}{2} \rceil+1}^{p}V_i^n.$$
Finally, we consider the flattening 
$$(V_1^{n_1})^{*}\rightarrow \bigotimes_{i=2}^{p}V_i^{n_i}$$
in the unbalanced case.
\end{proof}

\begin{Remark}\label{grass}
For Veronese varieties our results are equivalent to the identifiability criterion given by A. Iarrobino and V. Kanev in \cite{IK99}. In the $d$-factor Segre case they are weaker than reshaped Kruskal \cite[Proposition 16]{COV16} for $d$ odd but they perform better for $d$ even. While for unbalanced Segre our criteria perform better than \cite[Proposition 17]{COV16}.
\end{Remark} 

\begin{Remark}\label{bordR}
The algorithm in Proposition \ref{prop1gen} works for the border rank as well. Indeed, let $T$ be a tensor, and $P_t = U_{1,t}+\dots + U_{r,t}$, $Q_t = V_{1,t}+\dots + V_{r,t}$ be two sequence of rank $r$ tensors such that $\lim_{t\mapsto 0}P_t = \lim_{t\mapsto 0}Q_t = T$, and $\lim_{t\mapsto 0}\{U_{1,t},\dots,U_{r,t}\} \neq \lim_{t\mapsto 0}\{ V_{1,t},\dots,V_{r,t}\}$. Fix an $(A,B)$-flattening $\widetilde{T}:V_A^*\rightarrow V_B$ of $T$ such that $\dim(V_A^*)\geq r$, and let us denote by $\widetilde{U}_{i,t}, \widetilde{V}_{j,t}$, $\widetilde{P}_t$, $\widetilde{Q}_t$ the corresponding flattenings of $U_{i,t}, V_{j,t}, P_t, Q_t$. Then $\mathbb{P}(\widetilde{P}_t(V_A^*))\subseteq \left\langle \widetilde{U}_{1,t},\dots, \widetilde{U}_{r,t}\right\rangle$ and $\mathbb{P}(\widetilde{Q}_t(V_A^*))\subseteq \left\langle \widetilde{V}_{1,t},\dots, \widetilde{V}_{r,t}\right\rangle$ yield $\lim_{t\mapsto 0}\mathbb{P}(\widetilde{P}_t(V_A^*))\subset \Gamma_{U}$, $\lim_{t\mapsto 0}\mathbb{P}(\widetilde{Q}_t(V_A^*))\subset \Gamma_{V}$, where $\Gamma_U = \lim_{t\mapsto 0}\left\langle \widetilde{U}_{1,t},\dots, \widetilde{U}_{r,t}\right\rangle$ and $\Gamma_V = \lim_{t\mapsto 0}\left\langle \widetilde{V}_{1,t},\dots, \widetilde{V}_{r,t}\right\rangle$.

Now, let $X\subset\mathbb{P}(V_B)$ be the variety parametrizing rank one tensors. Since by hypothesis $\dim(\mathbb{P}(\widetilde{T}(V_A^*))) = r-1$ we have that $\mathbb{P}(\widetilde{T}(V_A^*))  = \lim_{t\mapsto 0}\mathbb{P}(\widetilde{P}_t(V_A^*)) = \lim_{t\mapsto 0}\mathbb{P}(\widetilde{Q}_t(V_A^*))$ forces $\mathbb{P}(\widetilde{T}(V_A^*)) = \Gamma_U = \Gamma_V$. Finally, since 
$$\lim_{t\mapsto 0}\{\widetilde{U}_{1,t},\dots,\widetilde{U}_{r,t}\}\subseteq X\cap \Gamma_U = X\cap \mathbb{P}(\widetilde{T}),\: \lim_{t\mapsto 0}\{\widetilde{V}_{1,t},\dots,\widetilde{V}_{r,t}\}\subseteq X\cap \Gamma_V = X\cap \mathbb{P}(\widetilde{T}(V_A^*))$$ and $\lim_{t\mapsto 0}\{\widetilde{U}_{1,t},\dots,\widetilde{U}_{r,t}\}\neq \lim_{t\mapsto 0}\{\widetilde{V}_{1,t},\dots,\widetilde{V}_{r,t}\}$ we get that $\deg(\mathbb{P}(\widetilde{T}(V_A^*))\cap X) \geq r+1$, a contradiction with hypothesis iii) of Proposition \ref{prop1gen}.
\end{Remark} 

Finally, we give an effective $7$-identifiability criterion for planes
quintics, and we extend it to the cases listed in Section \ref{intro}
when the uniqueness of decomposition holds for a general polynomial.

\begin{thm}\label{hi}
Let $F \in {\mathbb C}[x_{0},...,x_n]_{d}$ be a polynomial, and $H_{\partial,s}$ the linear span of its partial derivatives of order $s$ in ${\mathbb P}(k[x_0,...,x_n]_{d-s})$. 

Assume that:
\begin{itemize}
\item[-] $(n,d,h,s)\in \{(1,2h-1,h,h-2), (2,5,7,2), (3,3,5,1)\}$,
\item[-] $H_{\partial,s}$ has dimension $\binom{n+s}{n}-1$
\item[-] $H_{\partial,s}\cap V^n_{d-s}$ is empty.
\end{itemize}
Then $F$ is $h$-identifiable.
\end{thm}
\begin{proof}
Let us consider the case $(n,d,h,s)=(2,5,7,2)$. Assume that $F$ admits two different decompositions $\{L_{1},...,L_{7}\}$ and $\{l_{1},...,l_{7}\}$. Consider the second partial derivatives of $F$ and their span $H_{\partial}\subseteq \mathbb{P}^{9}$.  By Remark \ref{pd} a decomposition of $F$ induces a decomposition of its partial derivatives, hence we have 
$$H_{L} :=\langle L_{1}^{3},...,L_{7}^{3}\rangle\supset H_\partial\subset \langle l_{1}^{3},...,l_{7}^{3}\rangle=:H_l.$$ 
By hypothesis $\dim H_\partial=5$ and $H_\partial\cap V^3_2=\emptyset$, these yield:
\begin{itemize}
\item[i)] $H_\partial=H_L\cap H_l$,
\item[ii)] $L_i\neq l_j$ for any $i,j\in\{1,\ldots,7\}$,
\item[iii)] $H_L\cap V_3^2$ and $H_l\cap V_3^2$ are zero dimensional and $\sharp(H_L\cap V_3^2)=\sharp(H_l\cap V_3^2)=7$.
\end{itemize}

Let $H:=\langle H_L,H_l\rangle$ then $H$ intersects $V_{3}^{2}$ in at least 14 points and therefore 
$H\cap V_3^2$ contains a curve $\Gamma$ of degree  $3\gamma\leq 6$. 
Let $\Lambda$ be the pencil of hyperplanes containing $H$. Then $\Lambda_{|V^3_2}=\Gamma+\Sigma$, with $\Sigma$ a pencil of curves. Let $s$ be the degree of the base locus of $\Sigma$. The hypothesis $H_\partial\cap V^3_2=\emptyset$ and iii) yields 
$$s+6\gamma=14.$$
 On the other hand we only have the following possibilities:
\begin{itemize}
\item[-] $\gamma=1$ and $s=4$,
\item[-] $\gamma=2$ and $s=1$.
\end{itemize}
This contradiction proves the statement.

For 4-uples $(n,d,h,s)=(1,2h-1,h,h-2), (3,3,5,1)$ we may argue similarly to derive $h$-identifiability criteria, we leave the details to the reader.
\end{proof}

For some special values our methods yield a complete set of identifiability criteria.

\begin{Corollary}Let $V(n,d):=k[x_0,\ldots,x_n]_d$ be the vector space of homogeneous polynomial of degree $d$, with $k={\mathbb C},{\mathbb R}$.
Assume that the pair $(n,d)$ is in the following list
$$(1,d), (2,3), (2,4), (2,5), (2,6), (3,3), (3,4).$$ 
Then there is an effective criteria for specific $s$-identifiability for $V(n,d)$ for every $s$ where generic $s$-identifiability holds.
\end{Corollary}
\begin{proof} Let $k=\mathbb{C}$ be the complex field. For pairs $(1,d)$, $d$ odd, $(2,5), (3,3)$ we apply the identifiability conditions expressed in Theorem~\ref{hi} for the generic rank and Proposition~\ref{pro2gen} for subgeneric ranks. For $(2,4)$ Proposition~\ref{pro2gen} applies to ranks less then or equal to  4, and for rank 5 there is not generic identifiability due to defectivity. For $(3,4)$ Proposition~\ref{pro2gen} applies to ranks less than or equal to 6 and  Proposition~\ref{pro1.1gen} applies to rank 7, while rank 8 is not generically identifiable, \cite{COV15}.  For $(2,6)$ we apply  Proposition~\ref{pro2gen} for $s\leq 7$ and Proposition~\ref{pro1.1gen} for $s=8$, while rank 9  is not generically identifiable, due to weak defectivity \cite{COV15}.
  
To conclude we only need to extend the results  to the real field. For this let $F=\sum_1^{k_i} \lambda_iL_i^d$ be a real polynomial rank-1 decomposition. Then via a field extension we consider it over ${\mathbb C}$ and apply the criterion to prove complex and hence real identifiability.
\end{proof}

\subsection{Macaulay2 implementation}\label{mac2}
Finally, we implement our identifiability algorithms in Macaulay2 \cite{Mc2}. The package is in the ancillary file \texttt{Identifiability.m2}. After loading this package in Macaulay2, the main method available is \texttt{certifyIdentifiability}. 

The easiest ways to use this method are either by giving in input a mixed symmetric tensor $T$, represented by a multihomogeneous polynomial, and a positive integer $h$, or by inputting one of its decompositions $T=T_1+\dots+T_h$ into $h$ rank one mixed symmetric tensors. Then the method returns the boolean value \texttt{true} if the constraints of the correspondent $h$-identifiability criterion are satisfied for $T$. For more details we refer to the documentation (\texttt{viewHelp certifyIdentifiability}). In what follows we show how it works in some cases. 
{\normalsize 
\begin{verbatim}
Macaulay2, version 1.9.2
with packages: ConwayPolynomials, Elimination, IntegralClosure, LLLBases, 
               PrimaryDecomposition, ReesAlgebra, TangentCone

i1 : loadPackage "Identifiability";
--** Identifiability (v0.3) loaded **--
-- Example 1 -- Random degree 5 polynomial in 3 variables
i2 : P2 = QQ[x,y,z];
i3 : T = for i in 1..7 list (random(1,P2))^5;
i4 : time certifyIdentifiability(sum T,7)
-- got symmetric tensor of dimension 3 and degree 5
-- applying Theorem 3.7 (7-identifiability for 3-forms of degree 5)...
-- 7-identifiability certified
     -- used 0.257789 seconds
o4 = true
i5 : time certifyIdentifiability matrix{T}
-- got symmetric tensor of dimension 3 and degree 5
-- applying Theorem 3.7 (7-identifiability for 3-forms of degree 5)...
-- 7-identifiability certified
     -- used 0.228473 seconds
o5 = true
i6 : -- first 6 summands of T
     T' = T_{0..5};
i7 : time certifyIdentifiability(sum T',6)
-- got symmetric tensor of dimension 3 and degree 5
-- specific 6-identifiability certified
     -- used 0.0363902 seconds
o7 = true
i8 : time certifyIdentifiability matrix{T'}
-- got symmetric tensor of dimension 3 and degree 5
-- 6-identifiability certified
     -- used 0.0511795 seconds
o8 = true
-- Example 2 -- the command below creates a random mixed symmetric 
-- tensor of dimensions {2,5,4}, multidegree {3,2,3}, rank<=5 
i9 : T = multirandom({2,5,4},{3,2,3},5);
i10 : -- number terms of the tensor T
      # terms T
o10 = 1200
i11 : time certifyIdentifiability(T,5)
-- got mixed symmetric tensor of dimensions {2, 5, 4}
   and multidegree {3, 2, 3}
-- specific 5-identifiability certified
     -- used 4.54164 seconds
o11 = true
-- Example 3 -- Random 1 x 7 matrix of degree 4 polynomials in 4 variables
i12 : decomposition = multirandom'({4},{4},7);
i13 : time certifyIdentifiability decomposition
-- got symmetric tensor of dimension 4 and degree 4
-- applying Proposition 3.3...
-- 7-identifiability certified
     -- used 1.03492 seconds
o13 = true
-- Example 4 -- Random 1 x 8 matrix of degree 6 polynomials in 3 variables
i14 : decomposition = multirandom'({3},{6},8);
i15 : time certifyIdentifiability decomposition
-- got symmetric tensor of dimension 3 and degree 6
-- applying Proposition 3.3...
-- 8-identifiability certified
     -- used 0.440192 seconds
o15 = true
-- Example 5 -- Random degree 3 polynomial in 4 variables of rank<=5
i16 : F = multirandom({4},{3},5);
i17 : time certifyIdentifiability(F,5)
-- got symmetric tensor of dimension 4 and degree 3
-- applying Theorem 3.7 (5-identifiability for 4-forms of degree 3)...
-- 5-identifiability certified
     -- used 0.098442 seconds
o18 = true
-- Example 6 -- Random degree 69 polynomial in 2 variables
i19 : P1 = QQ[x,y];
i20 : F = random(69,P1);
i21 : time certifyIdentifiability(F,35)
-- got symmetric tensor of dimension 2 and degree 69
-- applying Theorem 3.7 (35-identifiability for 2-forms of degree 69)...
-- 35-identifiability certified
     -- used 469.406 seconds
o21 = true
\end{verbatim}
}\noindent 

\bibliographystyle{amsalpha}
\bibliography{Biblio}

\providecommand{\bysame}{\leavevmode\hbox to3em{\hrulefill}\thinspace}
\providecommand{\MR}{\relax\ifhmode\unskip\space\fi MR }
\providecommand{\MRhref}[2]{%
  \href{http://www.ams.org/mathscinet-getitem?mr=#1}{#2}
}
\providecommand{\href}[2]{#2}
\begin{thebibliography}{CGLM08}

\bibitem[AH95]{AH95}
J.~Alexander and A.~Hirschowitz, \emph{Polynomial interpolation in several
  variables}, J. Algebraic Geom. \textbf{4} (1995), no.~2, 201--222.
  \MR{1311347}

\bibitem[CC02]{CC02}
L.~Chiantini and C.~Ciliberto, \emph{Weakly defective varieties}, Trans. Amer.
  Math. Soc. \textbf{354} (2002), no.~1, 151--178. \MR{1859030}

\bibitem[CGLM08]{CGLM08}
P.~Comon, G.~Golub, L.~Lim, and B.~Mourrain, \emph{Symmetric tensors and
  symmetric tensor rank}, SIAM J. Matrix Anal. Appl. \textbf{30} (2008), no.~3,
  1254--1279. \MR{2447451}

\bibitem[CM96]{CM96}
P.~Comon and B.~Mourrain, \emph{Decomposition of quantics in sums of powers of
  linear forms}, Signal Processing \textbf{53} (1996), no.~2, 93--107.

\bibitem[COV15]{COV15}
L.~Chiantini, G.~Ottaviani, and N.~Vannieuwenhoven, \emph{On generic
  identifiability of symmetric tensors of subgeneric rank},
  \url{https://arxiv.org/abs/1504.00547v3}, to appear in Trans. Amer. Math.
  Soc. 2015.

\bibitem[COV16]{COV16}
\bysame, \emph{Effective criteria for specific identifiability of tensors and
  forms}, \url{https://arxiv.org/abs/1609.00123v1}, to appear in SIAM J. Matrix
  Anal. Appl. 2016.

\bibitem[DK93]{DK93}
I.~V. Dolgachev and V.~Kanev, \emph{Polar covariants of plane cubics and
  quartics}, Adv. Math. \textbf{98} (1993), no.~2, 216--301. \MR{1213725}

\bibitem[Dol04]{Do04}
I.~V. Dolgachev, \emph{Dual homogeneous forms and varieties of power sums},
  Milan J. Math. \textbf{72} (2004), 163--187. \MR{2099131}

\bibitem[GM16]{GM16}
F.~Galuppi and M.~Mella, \emph{Identifiability of homogeneous polynomials and
  cremona transformations}, \url{https://arxiv.org/abs/1606.06895v2}, 2016.

\bibitem[Hil88]{Hi88}
D.~Hilbert, \emph{Lettre adressée à m. hermite}, Journal de Mathématiques
  Pures et Appliquées (1888), 249--256 (fre).

\bibitem[IK99]{IK99}
A.~Iarrobino and V.~Kanev, \emph{Power sums, {G}orenstein algebras, and
  determinantal loci}, Lecture Notes in Mathematics, vol. 1721,
  Springer-Verlag, Berlin, 1999, Appendix C by Iarrobino and Steven L. Kleiman.
  \MR{1735271}

\bibitem[KB09]{BK09}
T.~G. Kolda and B.~W. Bader, \emph{Tensor decompositions and applications},
  SIAM Rev. \textbf{51} (2009), no.~3, 455--500. \MR{2535056}

\bibitem[Lan12]{La12}
J.~M. Landsberg, \emph{Tensors: geometry and applications}, Graduate Studies in
  Mathematics, vol. 128, American Mathematical Society, Providence, RI, 2012.
  \MR{2865915}

\bibitem[LO15]{LO15}
J.~M. Landsberg and G.~Ottaviani, \emph{New lower bounds for the border rank of
  matrix multiplication}, Theory Comput. \textbf{11} (2015), 285--298.
  \MR{3376667}

\bibitem[Mac92]{Mc2}
Macaulay2, \emph{Macaulay2 a software system devoted to supporting research in
  algebraic geometry and commutative algebra},
  \url{http://www.math.uiuc.edu/Macaulay2/}, 1992.

\bibitem[Mas16]{Ma16}
A.~Massarenti, \emph{Generalized varieties of sums of powers}, Bull. Braz.
  Math. Soc. (N.S.) \textbf{47} (2016), no.~3, 911--934. \MR{3549076}

\bibitem[Mel06]{Me06}
M.~Mella, \emph{Singularities of linear systems and the {W}aring problem},
  Trans. Amer. Math. Soc. \textbf{358} (2006), no.~12, 5523--5538. \MR{2238925}

\bibitem[Mel09]{Me09}
\bysame, \emph{Base loci of linear systems and the {W}aring problem}, Proc.
  Amer. Math. Soc. \textbf{137} (2009), no.~1, 91--98. \MR{2439429}

\bibitem[MM13]{MM13}
A.~Massarenti and M.~Mella, \emph{Birational aspects of the geometry of
  varieties of sums of powers}, Adv. Math. \textbf{243} (2013), 187--202.
  \MR{3062744}

\bibitem[MR13]{MR13}
A.~Massarenti and E.~Raviolo, \emph{The rank of {$n\times n$} matrix
  multiplication is at least {$3n^2-2\sqrt{2}n^\frac{3}{2}-3n$}}, Linear
  Algebra Appl. \textbf{438} (2013), no.~11, 4500--4509. \MR{3034546}

\bibitem[RS00]{RS00}
K.~Ranestad and F-O. Schreyer, \emph{Varieties of sums of powers}, J. Reine
  Angew. Math. \textbf{525} (2000), 147--181. \MR{1780430}

\bibitem[SU00]{SU00}
A.~Simis and B.~Ulrich, \emph{On the ideal of an embedded join}, J. Algebra
  \textbf{226} (2000), no.~1, 1--14. \MR{1749874}

\bibitem[Syl04]{Sy04}
J.~J. Sylvester, \emph{The collected mathematical papers}, vol.~1, Cambridge
  University Press, 1904.

\bibitem[TZ11]{TZ11}
H.~Takagi and F.~Zucconi, \emph{Spin curves and {S}corza quartics}, Math. Ann.
  \textbf{349} (2011), no.~3, 623--645. \MR{2755000}

\end{thebibliography}

\end{document}